\documentclass[draftclsnofoot,onecolumn]{IEEEtran}

\usepackage{amsmath,graphicx,amssymb,color} 
\usepackage{algorithmicx}
\usepackage{algpseudocode}
\usepackage{breqn}
\usepackage{cite}
\usepackage{amsthm}

\newtheorem{propo}{Proposition}

\DeclareMathOperator*{\argmin}{arg\,min}
\DeclareMathOperator*{\argmax}{arg\,max}

\usepackage{multirow}
\usepackage[caption=false,font=footnotesize,labelfont=footnotesize]{subfig}
\usepackage{bbm} 
\usepackage{mathtools} 
\usepackage{enumerate} 
\usepackage{flushend} 
\usepackage{stmaryrd} 
\DeclarePairedDelimiter{\ceil}{\lceil}{\rceil}
\DeclarePairedDelimiter{\floor}{\lfloor}{\rfloor}
\usepackage{cite}

\begin{document}

\title{Interpolation in the Presence of Domain Inhomogeneity}

\author{Hamid~Behjat, Zafer~Do\u{g}an, Dimitri~Van~De~Ville, and Leif~S\"ornmo
\thanks{
This work was supported by the Swedish Research Council under Grant 2009-4584, and in part by the Swiss National Science Foundation under grant P2ELP2-165160.
}
\thanks{
H. Behjat and L. S\"ornmo are with the Department of Biomedical Engineering, Lund University, Sweden (e-mail: hamid.behjat@bme.lth.se; leif.sornmo@bme.lth.se).
Z. Do\u{g}an is with the School of Engineering and Applied Sciences, Harvard University, USA (e-mail: zaferdogan@seas.harvard.edu ).
D. Van De Ville is with the \'Ecole Polytechnique F\'ed\'erale de Lausanne, Switzerland (e-mail: dimitri.vandeville@epfl.ch) and the Department of Radiology and Medical Informatics, University of Geneva.
}}


\maketitle

\begin{abstract}
Standard interpolation techniques are implicitly based on the assumption that the signal lies on a homogeneous domain. In this letter, the proposed interpolation method instead exploits prior information about domain inhomogeneity, characterized by different, potentially overlapping, subdomains. By introducing a domain-similarity metric for each sample, the interpolation process is then based on a domain-informed consistency principle. We illustrate and demonstrate the feasibility of domain-informed linear interpolation in 1D, and also, on a real fMRI image in 2D. The results show the benefit of incorporating domain knowledge so that, for example, sharp domain boundaries can be recovered by the interpolation, if such information is available. 
\end{abstract}


\IEEEpeerreviewmaketitle

\section{Introduction}
\IEEEPARstart{I}{nterpolation} has been studied extensively in various settings. The main frameworks are based on concepts such as smoothness for spline-generating spaces~\cite{Unser1999}, underlying Gaussian distributions for ``kriging''~\cite{Kriging1990}, and spatial relationship for inverse-distance-weighted interpolation~\cite{Lu2008}. Yet, while advanced concepts have been developed for describing these signal spaces, the underlying domain is always assumed to be homogeneous. In a sub-category of super-resolution image processing techniques \cite{Park2003}, such as in \cite{Alipour2010,  Cordero2012, Neubert2012, Manjon2010b, Rousseau2010, Jafari2014, JafariKhouzani2014,Rueda2013}, the interpolation phase is adapted based on the context of the signal; 
such adaptation schemes are based on the characteristics of either the image itself \cite{Alipour2010, Cordero2012, Neubert2012} or another high resolution image that is of the same nature as that of the low resolution image to be interpolated \cite{Rousseau2010, Manjon2010b, Jafari2014, JafariKhouzani2014,Rueda2013}. Here, we consider a different scenario in which signals are sampled over a known inhomogeneous domain; i.e., a domain characterized by a set of subdomains, that can be overlapping, available as supplementary data. This supplementary information is of a completely different nature than that of the samples to be interpolated; it describes the signal domain, rather than the signal itself.
 

\subsection*{Problem Formulation}
Assume that the following set of information is given:
\begin{enumerate}
\item
A sequence of samples $s[k]$ obtained as 
\begin{align}
\label{eq:samples}
s[k]  & = \left \langle s(x), \delta(x-k) \right  \rangle,\quad \text{for all }k \in \mathbb{Z},
\end{align}
where $s(\cdot) \in {L}_{2}$ (denoting the Hilbert space of all continuous, real-valued functions that are square integrable in Lebesgue sense) and $s[\cdot] \in \ell_{2}$ (denoting the Hilbert space of all discrete signals that are square summable).

\item Domain knowledge\footnote{We assume that every point belongs to at least one subdomain, and that the domain information can be specified  by a continuous function.
} described by J different subdomain indicator functions $\tilde{d}_j(x)$, $j=1,\ldots,J$. We then introduce the normalized subdomain functions as
\begin{equation}
\label{eq:constraintsPofUnity}
   d_j(x) = \frac{\tilde{d}_j(x)}{\sum_{l=1}^J \tilde{d}_{l} (x)}, \quad \text{such that } \sum_{j=1}^J d_j(x) = 1.
\end{equation}
Using $\tilde{d}_j(x)$, $j=1,\ldots,J$, space-dependent index sets of maximal and minimal association to the underlying subdomains can be derived as
\begin{align}
\label{eq:indSetH}
\mathcal{H}(x) & = \{ i | d_i(x) = \max_{j} d_j(x) \},\quad \text{for all } x \in \mathbb{R}, \\
\label{eq:indSetL}
\mathcal{L}(x) & =  \{ i | d_i(x) = \min_{j} d_j(x) \},\quad \text{for all } x \in \mathbb{R}.
\end{align}
\end{enumerate}

Given prior knowledge on domain inhomogeneity under the form (\ref{eq:constraintsPofUnity})--(\ref{eq:indSetL}), the objective is to adapt conventional interpolation methods such that they accommodate this information. To reach this objective, we start from shift-invariant generating kernels such as splines \cite{DeBoorBook,Unser1999,Unser2000B}, and then transform them into shift-variant kernels based on the domain knowledge. 



\subsection*{Potential Application Areas} 
In brain studies using functional magnetic resonance imaging (fMRI), a sequence of whole-brain functional data is acquired at relatively low spatial resolution. The data is commonly accompanied with a three to four fold higher resolution anatomical MRI scan, which provides information about the convoluted brain tissue delineating gray matter (GM) and white matter (WM), each of which have different functional properties \cite{Logothetis}, and cerebrospinal fluid (CSF); the topology also varies across subjects \cite{Mangin,deSchotten2011}. Hence, the goal is to exploit the richness of anatomical data to improve the quality of interpolation of fMRI data\cite{Brett2001, Andrade2001, Ashburner2007}, in the same spirit as approaches that aim to enhance denoising \cite{Kiebel2000, Behjat2014, Behjat2015} and decomposition \cite{Behjat2016} of fMRI data using anatomical data.



Earth sciences is another potential application area, where a spatially continuous representation of parameters, such as precipitation, land vegetation and atmospheric methane is desired to be computed from a discrete set of rain gauge measurements \cite{Borga1997,Lu2008}, fossil pollen measurements \cite{Gaillard2010, Pirzamanbein2014, Trondman2015} and satellite estimates of methane \cite{Frankenberg2005,Keppler2006}, respectively. In these scenarios, the well-defined geographical structure of the earth, anthropogenic land-cover models \cite{Kaplan2009} and geophysical models of the earth's surrounding atmosphere may be exploited as descriptors of the inhomogeneous domain to improve standard approaches to interpolation.

The remainder of this letter is organized as follows. In Section~II, the basis for obtaining a domain-informed interpolated signal is formulated. In Section~III, as a proof-of-concept, standard linear interpolation is extended to the proposed domain-informed setting, and an illustrative example is presented. In Section~IV, the proposed interpolation scheme is applied to a real fMRI image.

\section{Theory for Domain-Informed Interpolation}

The proposed domain-informed interpolation scheme is based on two fundamental concepts: (i)~deriving a domain-informed shift-variant basis, based on a shift-invariant basis of the integer shifts of a generating function $\varphi$; (ii)~fulfilling the ``domain-informed consistency principle''---a principle that we define as an extension of the consistency principle~\cite{Unser1994}.

\subsection{Domain-Informed Shift-Variant Basis}
\label{sec:disvb}
Consider a compactly-supported generator $\varphi(x)$ (i.e., $\varphi(x)=0$, $\forall |x| \ge \: \Delta^{(\varphi)} \in \mathbb{R}^{+}$) of a shift-invariant space 
\begin{equation}
\label{eq:representation0}
\mathcal{V}(\varphi) = \bigg\{ \hat{s}(x) = \sum_{k\in \mathbb{Z}} {c}[k] \cdot \varphi(x-k) : {c}[\cdot] \in \ell_{2} \bigg\},
\end{equation}
where $c[\cdot]$ are the weights of the integer-shifted basis functions. The generating function $\varphi(x)$ can be any of compact-support kernels used in standard interpolation. In the presence of domain inhomogeneity, the idea is to transform $\varphi(\cdot-k)$ into $\varphi_{\xi_k}(\cdot-k)$: a modulated version of $\varphi(\cdot-k)$ that is defined based on a domain similarity metric $\xi_{k}$ that describes the domain in the adjacency of $k$. With this construction, a shift-variant space
\begin{equation}
\label{eq:representation}
\mathcal{V}_\xi(\varphi) = \bigg\{ \hat{s}(x) = \sum_{k\in \mathbb{Z}} {c}[k] \cdot \varphi_{\xi_k}(x-k) : {c}[\cdot] \in \ell_{2} \bigg\},
\end{equation}
is obtained. We propose the following definition of $\xi_k$, which will subsequently guide the design of $\varphi_{\xi_k}$.

\subsubsection*{Definition (Domain Similarity Metric)}
\label{def:2}
Given a description of an inhomogeneous domain under the form (\ref{eq:constraintsPofUnity})--(\ref{eq:indSetL}), a domain similarity metric $ \xi_{k}(x) \in [0,1]$ can be defined in the $\Delta$ neighbourhood of each $k \in \mathbb{Z}$ as 
\begin{equation}
\label{eq:d}
\xi_{k}(x)= 
\begin{cases}
W_{k,x} \: \mathcal{S}\left (d_{h_{k}}(k+x) - 0.5 \right ), &|x| < \Delta, |\mathcal{H}(k)|=1, 
\\
W_{k,x} \: \mathcal{S}\left (d_{l_{k+x}}(k+x) - 0.5 \right ), & |x| < \Delta, |\mathcal{H}(k)|>1, 
\\
0, &|x| \ge \Delta,
\end{cases}
\end{equation}
where $|\mathcal{H}(.)|$ denotes the cardinality of set $\mathcal{H}(.)$, $\mathcal{S}(\cdot)$ denotes the logistic function, $ \mathcal{S}(n) = (1+e ^{-\gamma \: n} )^{-1} \in [0,1]$ with $\gamma >0$, $h_{.} \in \mathcal{H}(\cdot)$, $l_{.} \in \mathcal{L}(\cdot)$, and $W_{k,x}$ denotes a weight factor
\begin{equation}
\label{eq:weight}
W_{k,x} = 
1- \frac{1}{J} \sum_{j=1}^J \left | d_j(k+x) - d_j(k) \right |, 
\end{equation}
\noindent
which increases the adaptation to domain knowledge. For minimal adaptation, $W_{k,x}$ can be set to 1.

In (\ref{eq:d}), if $k \in \mathbb{Z}$ and $(k+x) \in \mathbb{R}$ are (i) maximally associated to the same subdomain, and (ii) the maximal association of $(k+x) \in \mathbb{R}$ has a probability greater than 0.5, $\xi_{k}(x) \rightarrow 1$, otherwise, $\xi_{k}(x) \rightarrow 0$. The parameter $\gamma$ of $\mathcal{S}(\cdot)$ tunes both the smoothness and strength of this adaptation; a greater $\gamma$, up to a suitable degree, leads to a stronger as well as smoother adaption to changes in domain inhomogeneity.

The metrics $\{\xi_{k}(x)\}_{k\in\mathbb{Z}}$ are defined as local functions in the neighbourhood of each  $k\in \mathbb{Z}$. On the global domain support, a domain similarity function, denoted  
$$
\rho(x,k): x \in \mathbb{R} \setminus \mathbb{Z}, k \in \mathcal{K}_{x}^{(\Delta)}:\{k \in \mathbb{Z}  \mid |x-k| < \Delta\} \rightarrow \llbracket 0,1 \rrbracket,
$$
can be defined as
\begin{equation}
\label{eq:S}
\rho(x,k) = \frac{\xi_{k}(x-k)}{\sum_{k'\in \mathcal{K}_{x}^{(\Delta)}} \xi_{k'}(x-k')},
\end{equation}
which satisfies the following three properties:
\begin{enumerate}
\item 
$\rho(x,m) = \rho(x,n)$ implies perfect similarity of the domains at $x$, $m$ and $n$. 
\item 
$\rho(x,m) > \rho(x,n)$ implies greater similarity of the domains at $x$ and $m$ than the similarity of the domains at $x$ and $n$, and vice versa.
\item for $x\in \mathbb{R}\setminus \mathbb{Z}$, we have $\sum_{k\in \mathbb{Z}} \rho(x,k) = 1$. See Appendix I for the proof. 
\end{enumerate}

\subsection{DICP: Domain-Informed Consistency Principle}
\label{sec:dicp}
There are different ways to define $\varphi_{\xi_k}$. In this letter, we consider the construction of a particular class of domain-informed, shift-variant basis that leads to interpolation satisfying the following principle.

\subsubsection*{Definition (Domain-Informed Consistency Principle)}
Given a sequence of samples, as in (\ref{eq:samples}), and a properly defined domain similarity function, as in (\ref{eq:S}), the interpolated signal $\hat{s}(\cdot)$ must satisfy the following conditions: 
\begin{enumerate}[(i)]
\item perfect fit at integers; i.e. $\hat{s}(k) = s[k],\quad \text{for all } k\in \mathbb{Z}$.
 
\item 
for any $x \in \{\mathbb{R} \setminus \mathbb{Z}\}$ and the set $\mathcal{K}_{x}^{(\Delta)}=\{ k\in \mathbb{Z}  \mid |x-k| < \Delta\}$, if for all $m,n\in \mathcal{K}_x^{(\Delta)}$ we have $\rho(x,m) \ne \rho(x,n)$,  
\begin{align}
\label{eq:criterion2}
\argmax_{k \in \mathcal{K}_{x}^{(\Delta)}} \Big \{ \rho(x,k) \Big \} 
= 
\argmin_{k \in \mathcal{K}_{x}^{(\Delta)}}  \Big \{ \big | \hat{s}(x) - s[k] \big | \Big \}.
\end{align}
\end{enumerate}

Criterion (i) is the consistency principle \cite{Unser1994}. Criterion (ii) is based on the assumption that the underlying signal $s(x)$, at any position $x \in \mathbb{R}\setminus \mathbb{Z}$, is more likely to be similar to those sample in its $\Delta$ neighbourhood that have a ``similar domain'' as $x$. As such, the DICP ensures that the interpolated signal is consistent not only at the given sample points, but also at intermediate points between samples. 
  
\section{DILI: Domain-Informed Linear Interpolation}
We propose a specific scheme to adapt standard linear interpolation (SLI) to incorporate domain knowledge. A definition of a shift-variant basis for this particular setting is presented, such that the interpolated signal is ensured to satisfy the DICP. In particular, the basis is a domain-informed version of the linear B-spline basis function~\cite{Unser1999}, the ``hat'' function with support $2$ ($
\Delta=1$) defined as 
\begin{equation}
\label{eq:firstDegSpline}
\Lambda(x) = 
\begin{cases}
1- |x|, & |x|<1
\\
0, & |x| \ge 1.
\end{cases}
\end{equation}  

\begin{propo} 
\label{propo:propo1} (Domain-Informed Linear Interpolation)
\noindent 
For a given set of samples (\ref{eq:samples}) and constraints (\ref{eq:constraintsPofUnity}), the domain-informed linear interpolated signal satisfying the DICP 
is given by
\begin{equation}
\label{eq:constrainedSplineInterp}
\forall x \in \mathbb{R}, \quad \hat{s}(x) = \sum_{k \in \mathbb{Z}} s[k]  \varphi_{\xi_k}(x-k),
\end{equation}
where 
\begin{align}
\label{eq:constrainedSpline}
\varphi_{\xi_k}(x) = 
\begin{cases}
0  & |x| \ge 1
\\
\Lambda(x) & |x| < 1, \: D(x_k) = 0, \: x_k \notin \mathcal{D}
\\
\tilde{\xi}_k(x) & \text{otherwise},
\end{cases}
\end{align}
where $x_k = (k+x) \in \mathbb{R}$, and 
\begin{align}
&\tilde{\xi}_k(x) = \frac{\xi_k(x)}{\xi_{\floor{x_k}}(x)+\xi_{\ceil{x_k}}(x)}, \text{ for }|x| < 1,
\label{eq:normDSM}
\\
&D(x) = \tilde{\xi}_{\floor{x}}(x-\floor{x}) -\tilde{\xi}_{\ceil{x}}(x-\ceil{x}), \text{ for } x\in\mathbb{R},
\nonumber
\\ 
& \mathcal{D} = \Big \{ x \in \mathbb{R} \: \big | \:  D(x) = 0, \: \lim_{\varepsilon \rightarrow 0} \int_{x-\varepsilon}^{x+\varepsilon} D(x)  d x = 0 \Big \},
\nonumber
\end{align}
\end{propo}

\begin{proof}
See Appendix II.
\end{proof}

\noindent
DILI has the property that in any domain interval $\llbracket \alpha,\beta \rrbracket$, $\alpha,\beta \in \mathbb{R}$, that is either homogeneous, i.e.,  
$$\forall x \in \llbracket \alpha,\beta \rrbracket : d_{l}(x) = 1, \: \: d_{j}(x) = 0, \: \: j \in \{\{1,\ldots,J\} - l\},$$
or uniformly inhomogeneous, i.e., 
$$
\forall x \in \llbracket \alpha,\beta \rrbracket :  d_{1}(x) = d_{2}(x) = \cdots = d_{J}(x),
$$
it exploits basis functions that are identical to those used in SLI, i.e., $ \forall  x \in \llbracket \alpha,\beta \rrbracket, k \in \mathbb{Z} :  \varphi_{k}(x) = \Lambda(x-k)$; thus, the DILI and SLI signals are identical within the interval $\llbracket \alpha,\beta \rrbracket$.

\begin{figure}[]
\centering
\includegraphics[width=0.77\textwidth]{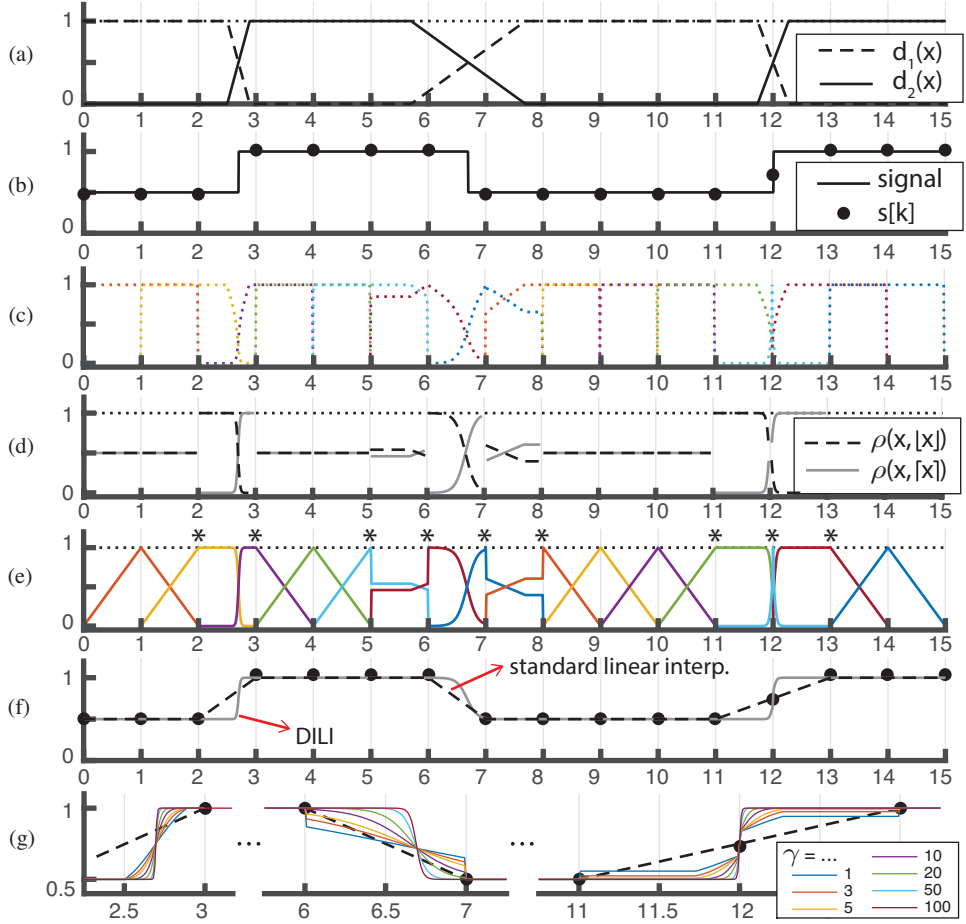}
\caption{
DILI. 
(a) Signal domain. (b) Signal samples. 
(c) $\{\xi_{k}(x)\}_{k=1,\ldots14}$, for $|x|<=1$. 
(d) $\{\rho(x,k)\}_{k \in \{\floor{x},\ceil{x}\}}$, for $\Delta=1$. 
(e) $\{\varphi_{\xi_k}(x)\}_{k=1,\ldots,14}$. 
In (a), (d)-(e), the black dotted lines show partition of unity constraints. (f) DILI vs. SLI. (g) DILI using a range of different $\gamma$. The dashed line corresponds to SLI as shown in (f).
\vspace{-3mm}
}
\label{fig:results}
\end{figure}

An example setting for constructing DILI is presented in Fig.~\ref{fig:results}. The signal domain consists of two subdomains, see Fig.~\ref{fig:results}(a), that satisfy (\ref{eq:constraintsPofUnity}). The domain has several homogeneous intervals, such as $\llbracket 0,2.5\rrbracket$ or $\llbracket 8,11\rrbracket$, as well as inhomogeneous intervals. In particular, three types of inhomogeneous domain intervals are observed at the transition between the two subdomains: a fast transition that falls between two samples, i.e., interval $\llbracket 2,3\rrbracket$, a slow transition, i.e. interval $\llbracket 5,8\rrbracket$, and a transition that occurs symmetrically around a sample point, $k=12$. The signal samples and the underlying continuous signal are displayed in Fig.~\ref{fig:results}(b). 

Fig.~\ref{fig:results}(c) illustrates the set of domain similarity metrics $\{\xi_{k}(x)\}_{k\in \mathbb{Z}}$, in $\Delta=1$ neighbourhood of each $k\in Z$. Fig.~\ref{fig:results}(d) illustrates the corresponding domain similarity function $\{\rho(x,k)\}_{k \in \mathcal{K}_{x}^{(1)}}$, defined over the global support; the function is displayed in two parts, defining the left-hand and right-hand local neighbourhood of each $k\in \mathbb{Z}$. Fig.~\ref{fig:results}(e) illustrates the resulting set of domain-informed splines, cf. (\ref{eq:constrainedSpline}); only those that reside in the adjacency of the domain transition boundaries, marked with asterisks, deviate from $\Lambda(x)$. A greater number deviate from their standard counterpart at the slower subdomain transition interval. As sample $s[12]$ lies at the exact intersection of the two subdomains, $\xi_{12}(x)$ as well as $\varphi_{k}(x)$ are strongly suppressed; i.e., $s[12]$ will have minimal effect in the interpolated values in its neighbourhood. 

The resulting SLI and DILI are shown in Fig.~\ref{fig:results}(f). SLI and DILI are identical in regions associated to only a single subdomain. However, DILI better matches the underlying signal in the subdomain transition bands, satisfying the DICP. In Fig.~\ref{fig:results}(c)-(f), the logistic function $\mathcal{S}(\cdot)$ in (\ref{eq:d}) was used with parameter $\gamma=20$. Fig.~\ref{fig:results}(g) illustrates the effect of varying $\gamma$; the DILI signals are illustrated only in the adjacency of the three subdomain transition regions, since they are, elsewhere, identical to SLI. In essence, $\gamma$ determines the strength of associating the point to be interpolated, to the similarity of its domain and that of nearby samples; a greater $\gamma$ results in a greater strength of this relationship. 

\begin{figure}[]
\centering
\includegraphics[width=0.73\textwidth]{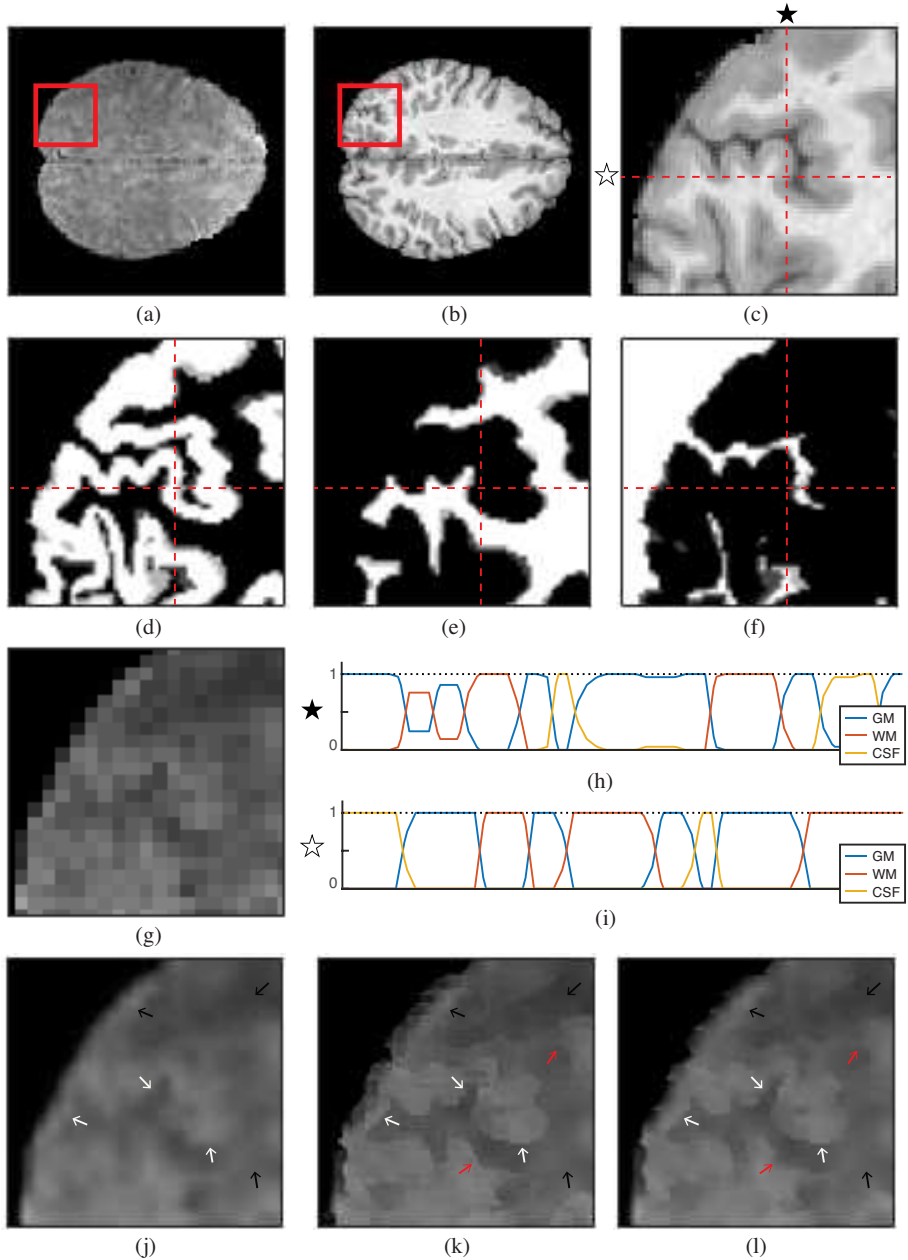}
\caption{DIBLI. (a) A 2D slice of an fMRI volume, (b) The structural scan. (c) Close-up of an ROI. (d) GM, (e) WM, and (f) CSF of the ROI. (g) fMRI data in the ROI. (h)-(i) Column and row domain data for the marked position in (c). (j) SBLI. (k) DIBLI with $W_{k,x}$ as in (8). (l) DIBLI with $W_{k,x}=1$. 
\vspace{-3mm}
}
\label{fig:dibli}
\end{figure}

\section{DIBLI: Domain-Informed Bilinear Interpolation}
\label{sec:dibli}
Domain-informed bilinear interpolation (DIBLI), can be obtained as a direct, separable extension of DILI to 2-D space. Fig.~\ref{fig:dibli} illustrates the setting for applying DIBLI on a 2-D slice of an fMRI volume, see Fig.~\ref{fig:dibli}(a), that accompanies a 3-fold higher resolution structural scan, see Fig.~\ref{fig:dibli}(b). Segmenting the structural scan, one obtains GM, WM, and CSF probability maps, see Figs.~\ref{fig:dibli}(d)-(f); these maps can be treated as normalized subdomain functions that satisfy (\ref{eq:constraintsPofUnity}) across any column/row in the plane, see Figs.~\ref{fig:dibli}(h)-(i). Fig.~\ref{fig:dibli}(j) shows standard bilinear interpolation (SBLI) of the functional pixels shown in Fig.~\ref{fig:dibli}(g). Figs.~\ref{fig:dibli}(k)-(l) show two versions of DIBLI, the former with maximal and the later with minimal adaptation to domain knowledge. SBLI and both DIBLI versions are identical at homogeneous parts of the domain (see black arrows), whereas at the inhomogeneous parts (see white arrows), both DIBLI versions present finer details. 
DIBLI with maximal adaptation provides further details over the minimal adapted version at some parts (see red arrows). Overall, DIBLI with minimal adaptation, cf. Fig.~\ref{fig:dibli}(l), may seem more visually appealing than Fig.~\ref{fig:dibli}(k), and yet, 
it presents significant subtle details that are missing in SBLI.

\section{Conclusion and Outlook}
We have proposed an interpolation scheme that incorporates a-priori knowledge of the signal domain, such that the interpolated signal is consistent not only at sample points, with respect to the given samples, but also at intermediate points between samples, with respect to the given domain knowledge. As a proof-of-concept, domain-informed linear interpolation has been presented as an extension of standard linear interpolation. Interpolation approaches that use higher order B-splines may also be extended based on the proposed domain-informed consistency principle, by defining suitable domain similarity metrics matching the support of the generating kernel.    
Results from applying the proposed approach on fMRI data demonstrated its potential to reveal subtle details; our future research will focus on further evaluation of its properties as well as its efficient implementation.

\section*{Appendix I}
Expanding the sum of integer-shifted domain similarity metric functions, $\forall x \in \mathbb{R} / \mathbb{Z}$, we have   
\begin{align}
\label{eq:proof_POU}
\sum_{k \in \mathbb{Z}} \rho(x,k) 
& \stackrel{(\ref{eq:d})}{=} \cdots + 0 + \frac{ \sum_{i\in \mathcal{K}_{x}^{(\Delta)}} \xi_{i}(x-i)}{\sum_{k'\in \mathcal{K}_{x}^{(\Delta)}} \xi_{k'}(x-k')} + 0 + \cdots =1.
\end{align}

\vspace{-7mm}
\section*{Appendix II}
\begin{proof} (Proposition 1) We prove that the proposed interpolation satisfies both conditions of the DICP, cf. Definition~1. 

\subsubsection*{Condition (i)}
Perfect fit at integers is satisfied since    
\begin{align}
\forall k \in \mathbb{Z}, \quad \hat{s}(k)  & = \left \langle \hat{s}(x), \delta(x-k) \right \rangle 
\nonumber
\\
& \stackrel{(\ref{eq:constrainedSplineInterp})}{=} \sum_{n \in \mathbb{Z}} s[k] \underbrace{\left \langle \varphi_{\xi,n}(x-n), \delta(x-k) \right \rangle}_{\varphi_{\xi,n}(k-n)} 
\nonumber
\\
& \: \, = s[k] \underbrace{\varphi_{\xi,k}(0)}_{\stackrel{(\ref{eq:constrainedSpline}),(\ref{eq:d})}{=} 1\: \: } =  s[k].
\end{align}

\subsubsection*{Condition (ii)}
Define $\forall x \in \{\mathbb{R} \setminus \mathbb{Z}\}$, $f_{\floor{x}}(x) = | \hat{s}(x) - s[{\floor{x}}] |$ and $f_{\ceil{x}}(x) = | \hat{s}(x) - s[{\ceil{x}}] |$.
The second condition of the principle, cf. (\ref{eq:criterion2}), is satisfied if it can be shown that:
\begin{align}
\rho(x,\floor{x}) & > \rho(x,\ceil{x}) \rightarrow f_{\floor{x}}(x) < f_{\lceil x \rceil}(x),
\label{eq:toProve1}
\\
\rho(x,\floor{x}) & < \rho(x,\ceil{x}) \rightarrow f_{\floor{x}}(x) > f_{\lceil x \rceil}(x).  
\label{eq:toProve2}
\end{align}
Define $ \forall x \in \mathbb{R}, x^{+} = x-\floor{x}$ and $x^{-} = x-\ceil{x}$. Integer-shifted $\Lambda(x)$ form a partition of unity, $\forall x \in \mathbb{R}$, as 
\begin{align}
\label{eq:proof_POUsplines}
\sum_{k \in \mathbb{Z}} \Lambda(x-k) = \Lambda(x^{+}) + \Lambda(x^{-}) =1,
\end{align}
and so do integer-shifted $\{\tilde{\xi}_{k}(x)\}_{k\in \mathbb{Z}}$, cf.~(\ref{eq:normDSM}), as
\begin{equation}
\label{eq:dsm_POU}
\sum_{k \in \mathbb{Z}} \tilde{\xi}_{k}(x-k) \stackrel{(\ref{eq:proof_POU})}{=} \tilde{\xi}_{\floor{x}}(x^{+}) + \tilde{\xi}_{\ceil{x}}(x^{-}) =1.
\end{equation} 
Thus, integer-shifted, domain-informed first degree spline basis form a partition of unity, since $\forall x \in \mathbb{R}$    
\begin{align}
\label{eq:proof_POUcnstrainedSplines}
\sum_{k \in \mathbb{Z}} \varphi_{\xi,k}(x-k) \stackrel{(\ref{eq:constrainedSpline})}{=} \varphi_{\xi,\floor{x}}(x^{+}) + \varphi_{\xi,\ceil{x}}(x^{-}) \stackrel{(\ref{eq:dsm_POU}),(\ref{eq:proof_POUsplines})}{=} 1.
\end{align}
The function $f_{\floor{x}}(x)$ in (\ref{eq:toProve1}) can be expanded as 
\begin{align}
f_{\floor{x}}(x)
& \stackrel{(\ref{eq:constrainedSplineInterp})}{=} 
\bigg | 
\sum_{k=\floor{x}}^{\ceil{x}} s[k] \varphi_{\xi,k}(x-k) - s[\floor{x}]
\bigg |
\nonumber
\\
& \: =  
\big | 
s[\floor{x}] \underbrace{\left ( \varphi_{\xi,\floor{x}}(x^{+}) -1 \right)}_{\stackrel{(\ref{eq:proof_POUcnstrainedSplines})}{=} - \varphi_{\xi,\ceil{x}}(x^{-})} + s[\ceil{x}] \, \underbrace{\varphi_{\xi,\ceil{x}}(x^{-})}_{ \stackrel{(\ref{eq:constrainedSpline}),(\ref{eq:d})}{\longrightarrow}  \ge 0 \: \quad  \quad}
\big |
\nonumber
\\
& \: =  
\big | 
s[\ceil{x}] - s[\floor{x}] \big | \cdot \varphi_{\xi,\ceil{x}}(x^{-}).
\label{eq:prf1}
\end{align}

\noindent 
Similarly, $f_{\lceil x \rceil}(x)$ in (\ref{eq:toProve1}) can be written as 
\begin{equation}
f_{\lceil x \rceil}(x) = s[\floor{x}] - s[\ceil{x}] \big | \cdot \varphi_{\xi,\floor{x}}(x^{+}). 
\label{eq:prf2}
\end{equation}

\noindent
From the left hand relation in (\ref{eq:toProve1}) we have 
\begin{align}
\label{eq:prf0}
\rho(x,\floor{x})  > \rho(x,\ceil{x})  \stackrel{(\ref{eq:S}), (\ref{eq:constrainedSpline})}  {\longrightarrow} \varphi_{\xi,\floor{x}}(x^{+}) & > \varphi_{\xi,\ceil{x}}(x^{-})
\nonumber
\\
& \downarrow (\ref{eq:prf1}), (\ref{eq:prf2})
\nonumber
\\
f_{\ceil{x}}(x) & > f_{\floor{x}}(x).
\end{align}
(\ref{eq:toProve2}) can be proved in the same way as done for (\ref{eq:toProve1}).    
\end{proof}

\bibliographystyle{IEEEtran}
\bibliography{hbehjat_bibligraphy}

\end{document}